\documentclass[11pt]{amsart}

\DeclareMathOperator{\Lie}{Lie}\DeclareMathOperator{\im}{im}

\DeclareMathOperator{\rk}{rank}
\DeclareMathOperator{\ad}{ad}

\usepackage[usenames,dvipsnames]{color}

\usepackage{amsmath,amssymb,amsthm,graphics,amscd}
\usepackage{longtable,pifont}
\usepackage{cite}
\usepackage{tikz}
\usepackage{color}
\usepackage{graphicx}
\usepackage{epsf}
\usepackage{fancyhdr}
\usepackage{wrapfig}

\usepackage{hyperref}
\hypersetup{
colorlinks=true,citecolor=purple}

  \renewenvironment{thebibliography}[1]{
    \begin{oldthebibliography}{#1}
      \setlength{\parskip}{0ex}
      \setlength{\itemsep}{0ex}
  }
  {
    \end{oldthebibliography}
  }

\usepackage[capitalize]{cleveref}
\setlongtables
\begin{document}

\newcommand{\cmark}{\ding{51}}
\newcommand{\xmark}{\ding{55}}

\newcounter{rownum}
\setcounter{rownum}{0}
\newcommand{\ab}{\addtocounter{rownum}{1}\arabic{rownum}}

\newcommand{\x}{$\times$}
\newcommand{\bb}{\mathbf}

\newcommand{\Ind}{\mathrm{Ind}}
\newcommand{\Char}{\mathrm{char}}
\newcommand{\hra}{\hookrightarrow}

\newtheorem{lemma}{Lemma}[section]
\newtheorem{theorem}[lemma]{Theorem}
\newtheorem*{claim}{Claim}
\newtheorem{cor}[lemma]{Corollary}
\newtheorem{conjecture}[lemma]{Conjecture}
\newtheorem{Proposition}[lemma]{Proposition}
\newtheorem{question}[lemma]{Question}

\theoremstyle{definition}
\newtheorem{example}[lemma]{Example}
\newtheorem{examples}[lemma]{Examples}
\newtheorem{algorithm}[lemma]{Algorithm}
\newtheorem*{algorithm*}{Algorithm}

\theoremstyle{remark}
\newtheorem{remark}[lemma]{Remark}
\newtheorem{remarks}[lemma]{Remarks}
\newtheorem{obs}[lemma]{Observation}

\theoremstyle{definition}
\newtheorem{defn}[lemma]{Definition}

  \def\hal{\unskip\nobreak\hfil\penalty50\hskip10pt\hbox{}\nobreak
  \hfill\vrule height 5pt width 6pt depth 1pt\par\vskip 2mm}

\renewcommand{\labelenumi}{(\roman{enumi})}

\def\bar{\overline}
\def\sub{\subseteq}
\def\iso{\cong}
\def\isoto{\overset{\sim}{\longrightarrow}}
\def\longto{\longrightarrow}

\def\kk{\Bbbk}

\def\ad{\operatorname{ad}}
\def\Ad{\operatorname{Ad}}
\def\Ext{\operatorname{Ext}}
\def\Cl{\operatorname{Cl}}
\def\Lie{\operatorname{Lie}}
\def\rad{\operatorname{rad}}
\def\sspan{\operatorname{span}}
\def\height{\operatorname{ht}}

\def\C{{\mathbb C}}
\def\F{{\mathbb F}}
\def\Z{{\mathbb Z}}
\def\G{{\mathbb G}}

\def\GL{\mathrm{GL}}
\def\SL{\mathrm{SL}}
\def\Sp{\mathrm{Sp}}
\def\SO{\mathrm{SO}}
\def\O{\mathrm{O}}
\def\PGL{\mathrm{PGL}}
\def\PSL{\mathrm{PSL}}

\def\b{\mathfrak b}
\def\g{\mathfrak g}
\def\h{\mathfrak h}
\def\l{\mathfrak l}
\def\s{\mathfrak s}
\def\p{\mathfrak p}
\def\q{\mathfrak q}
\def\t{\mathfrak t}
\def\u{\mathfrak u}
\def\m{\mathfrak m}
\def\n{\mathfrak n}
\def\c{\mathfrak c}
\def\X{\mathfrak X}

\def\gl{\mathfrak{gl}}
\def\sl{\mathfrak{sl}}
\def\sp{\mathfrak{sp}}
\def\so{\mathfrak{so}}
\def\pgl{\mathfrak{pgl}}

\def\cN{\mathcal N}
\def\cO{\mathcal O}
\def\cU{\mathcal U}
\def\cV{\mathcal V}
\def\cX{\mathcal X}
\def\cf{\mathcal F}

\def\oe{\overline{e}}
\def\oh{\overline{h}}
\def\of{\overline{f}}

\newenvironment{changemargin}[1]{%
  \begin{list}{}{%
    \setlength{\topsep}{0pt}%
    \setlength{\topmargin}{#1}%
    \setlength{\listparindent}{\parindent}%
    \setlength{\itemindent}{\parindent}%
    \setlength{\parsep}{\parskip}%
  }%
  \item[]}{\end{list}}

\parindent=0pt
\addtolength{\parskip}{0.5\baselineskip}

\title{Monogamous subvarieties of the nilpotent cone}

\author{Simon M. Goodwin}
\address{School of Mathematics, University of Birmingham, Birmingham, B15 2TT, UK} \email{s.m.goodwin@bham.ac.uk {\text{\rm(Goodwin)}}}

\author{Rachel Pengelly}
\address{Department of Mathematics, The University of Manchester, Oxford Road, Manchester, M13 9PL, UK}
\email{ rachel.pengelly@manchester.ac.uk {\text{\rm(Pengelly)}}}

\author{David I. Stewart}
\address{Department of Mathematics, The University of Manchester, Oxford Road, Manchester, M13 9PL, UK}
\email{ david.i.stewart@manchester.ac.uk {\text{\rm(Stewart)}}}

\author{Adam R. Thomas} 
\address{Department of Mathematics, University of Warwick, Coventry, CV4 7AL, UK} 
\email{adam.r.thomas@warwick.ac.uk {\text{\rm{(Thomas)}}}}
\dedicatory{In memory of Gary, who influenced us greatly}

\pagestyle{plain}
\begin{abstract} Let $G$ be a reductive algebraic group over an algebraically closed field $\kk$ of prime characteristic not $2$, whose Lie algebra is denoted $\g$. We call a subvariety $\X$ of the nilpotent cone $\cN\subset \g$ \emph{monogamous} if for every $e\in \X$, the $\sl_2$-triples $(e,h,f)$ with $f\in \X$ are conjugate under the centraliser $C_G(e)$. Building on work by the first two authors, we show there is a unique maximal closed $G$-stable monogamous subvariety $\cV\subset \cN$ and that it is an orbit closure, hence irreducible. We show that $\cV$ can also be characterised in terms of Serre's $G$-complete reducibility.
\end{abstract}
\maketitle

\section{Introduction}

Let $\kk$ be an algebraically closed field of characteristic $p \neq 2$, and $G$ a simple algebraic $\kk$-group  with Lie algebra $\g = \Lie (G)$. Three elements $e,h,f \in \g$ form an $\sl_2$-triple if the subalgebra $\left< e,h,f \right>$ is a homomorphic image of $\sl_2(\kk)$. That is, $(e,h,f)$ satisfy the relations\footnote{When the characteristic is two these relations degenerate leading to a qualitatively different theory; see \cite{SteTh24} for more details. This justifies our underlying assumption of $p \neq 2$.}
\[ [h,e]=2e, \quad [h,f]=-2f, \quad [e,f]=h. \]
Theorems of Jacobson--Morozov and Kostant say that if $\kk$ is of characteristic $0$, then for any nilpotent $e \in \g$ there exists an $\sl_2$-triple $(e,h,f)$ in $\g$ which is unique up to conjugacy by the centraliser of $e$ in $G$, see \cite{Morozov,Jacobson,Kostant}. 

Over fields of positive odd characteristic, for any nilpotent $e \in \g$ there exists an $\sl_2$-triple $(e,h,f)$ in $\g$ except in the case $G$ is of type $G_2$, $p=3$, and $e$ is in the $\tilde{A_1}_{(3)}$ class \cite[Theorem 1.7]{ST18}. We continue the investigation into generalising Kostant's uniqueness theorem to fields of small characteristic. Let $\X$ be a subset of the nilpotent cone $\cN \subset \g$. We say that $\X$ is \emph{monogamous} if the following property holds: 

\begin{center}
{Let $(e,h,f)$ and $(e,h',f')$ be $\sl_2$-triples with $e,f, f' \in \X$. Then $(e,h,f)$ is $C_G(e)$-conjugate to $(e,h',f')$.}
\end{center}

The main theorem of \cite{ST18} proves that $\cN$ is monogamous if and only if $p > h(G)$, where $h(G)$ is the Coxeter number for $G$. When $G$ is of classical type, the first two authors \cite{GP} showed that there exists a unique maximal $G$-stable closed subvariety of $\cN$ that is monogamous, and give an explicit description of these. This paper completes the story by treating the exceptional types. Define the following subset of $\cN$:  
\[ \cV  := \left\{ x \in \cN \  \left| \  \begin{array}{l} 
 x^{[p]}=0, \ \\
x \text{ is not regular in a Levi subalgebra with a factor of type } A_{p-1}, \text{ and} \\
x \text{ is not subregular if } G \text{ is of type } G_2 \text{ and } p=3.
\end{array}   \right\}\right. \]

\begin{theorem} \label{thm:mainKos}
Let $G$ be a simple algebraic group over an algebraically closed field $\kk$ of characteristic $p > 2$. Then $\cV$ is the unique maximal $G$-stable closed monogamous subvariety of $\cN$. Furthermore, $\cV$ is irreducible, being the closure of a single orbit as specified in Tables \ref{tab:closurecVclass} and \ref{tab:closurecVexcep} below.
\end{theorem}

In \cite{ST18}, a close relationship was found between uniqueness of $\sl_2$-subalgebras and the existence of so-called non-$G$-cr $\sl_2$-subalgebras. The notion of $G$-complete reducibility for subgroups of $G$ is due to Serre \cite{Ser05}, and the natural generalisation to subalgebras of $\g$ was introduced by McNinch \cite{McN07}.  Given a subalgebra $\h \subseteq \g$, we say that $\h$ is \emph{$G$-completely reducible} ($G$-cr for short) if for every parabolic subalgebra $\p$ such that $\h \subseteq \p$ there exists some Levi subalgebra $\l$ of $\p$ with $\h \subseteq \l$.

 We say $\X \subseteq \cN$ is \emph{$A_1$-$G$-cr} if every subalgebra generated by an $\sl_2$-triple $(e,h,f)$ with $e,f \in \X$ is $G$-cr.   

\begin{theorem} \label{thm:mainGcr}
Let $G$ be a simple algebraic group over an algebraically closed field $\kk$ of characteristic $p > 2$. Then $\cV$ is the unique maximal $G$-stable closed $A_1$-$G$-cr subvariety of $\cN$.
\end{theorem}

The proof follows very quickly from Theorem \ref{thm:mainKos}; see Section \ref{sec:thmproofs}. 

\begin{remark} 
  It would be interesting to know more about the geometry of the nilpotent variety $\cV$. In type $A$, Donkin \cite{Donk90} showed that the closure of each orbit is normal. Orbit closures in the remaining classical types are considered by Xiao and Shu \cite{XIAO201533}. For exceptional types $G_2, F_4, \ldots, E_8$, results of Thomsen \cite{THOMSEN} show that our varieties $\cV$ are in fact Gorenstein normal varieties with rational singularities as long as $p \geq 5, 11, 7, 11, 13$, respectively. 
\end{remark}

\subsection*{Acknowledgments}
Part of this work contributed to the second author’s PhD thesis at the University of Birmingham, they were supported by the EPSRC during this period. The second author also gratefully acknowledges the financial support of both the LMS and the Heilbronn institute. The third author is supported by a Leverhulme Trust Research Project Grant RPG-2021-080 and the fourth author is supported by an EPSRC grant EP/W000466/1. The authors thank the anonymous referee for their careful reading and numerous suggestions that have improved the paper. For the purpose of open access, the authors have applied a Creative Commons Attribution (CC BY) licence to any Author Accepted Manuscript version arising from this submission.

\section{Preliminaries}

Throughout, $\kk$ is an algebraically closed field of characteristic $p>2$ and $G$ is a simple $\kk$-group with $\g = \Lie(G)$. There is an inherited $[p]$-map on $\g$ and we use $x^{[p]}$ to denote the image of $x \in \g$ under this map.  The variety of all nilpotent elements in $\g$, often called the nilpotent cone, is denoted by $\cN$. The restricted nullcone is the subvariety of $\cN$ consisting of elements $x$ such that $x^{[p]} = 0$ and we denote it by $\cN_p$. The distribution of nilpotent elements among $\sl_2$-subalgebras of $\g$ is insensitive to central isogeny, and so we assume that whenever $G$ is classical, it is one of $\SL(V)$, $\Sp(V)$ or $\SO(V)$ and write $G=\Cl(V)$ for brevity; if $G$ is exceptional, we take it to be simply connected. 

Recall that a prime $p$ is bad for $G$ if $p=2$ and $G$ is of type $B$, $C$ or $D$; if $p\leq 3$ and $G$ is exceptional; or if $p\leq 5$ and $G$ is of type $E_8$; otherwise it is good. In some examples we require a choice of base for the root system associated to $\g$; we use Bourbaki notation \cite{Bourb05}. Finally, we fix a maximal torus $T$ of $G$.

\subsection{Nilpotent orbits and Hasse diagrams} \label{ss:nilporbits}

The orbits for the action of $G$ on $\cN$ are called nilpotent orbits. There are finitely many such and they are classified. In case $G$ is of exceptional type, we describe an orbit $\O=G\cdot x$ by a label indicating a Levi subalgebra in which $e$ is distinguished; for these labels we refer to \cite{LS12}. 

When $G=\Cl(V)$, the classification of orbits in terms of the action on $V$ is well-known and can be found in \cite[Section 1]{Jan04}, but we recap it here for ease of reference. Set $m=\dim V$. If $G=\SL(V)$, orbits are parameterised by partitions of $m$ according to the Jordan decomposition of their elements' actions on $V$; we write $x\sim (\lambda_1,\dots,\lambda_r)$ where $\lambda_1\geq \dots\geq \lambda_r$ is the partition of $m$ corresponding to $x$. In types $B$ and $C$ orbits are parameterised by partitions of $m$ with an even number of even parts and an even number of odd parts, respectively. In type $D$ it is slightly more complicated. A partition is called very even if it only has even parts and they all occur with even multiplicity. There is one orbit for each partition of $m$ with an even number of even parts that is not very even; and two orbits for each very even partition of $m$. 

To check that $\cV$ is a closed subvariety of $\cN$ we require information about the Hasse diagrams for the closure relation on nilpotent orbits. For classical types, apart from type $D$, the closure order on orbits is precisely the dominance order on partitions. In type $D$ we start with the Hasse diagram for the dominance order on partitions with an even number of even parts. Then we replace each very even partition $\lambda$ with two nodes $\lambda_1, \lambda_2$ and replace each edge from $\lambda$ to $\mu$ with two edges from $\lambda_i$ to $\mu$. For exceptional types the picture is actually incomplete in general. But if $p$ is good for $G$, the existence of Springer morphisms implies that the Hasse diagrams remain the same as those in characteristic $0$; \cite[Th\`eor\'eme III 5.2]{Spa82}. These can be found in \cite[pp.247--250]{Spa82} and are reproduced in \cite[Section~13.4]{Car93} with labels closer to those in \cite{LS12}. However, those in \cite{Car93} are missing edges in the $E_6, E_7$ and $E_8$ diagrams. Specifically, there should be an edge between the following pairs of labels:   

\hspace{1cm} $E_6$:  $(D_4(a_1), A_3)$, 

\hspace{1cm} $E_7$: $(D_6(a_2), D_5(a_1) + A_1)$, $(D_5(a_1), D_4)$, $(D_4(a_1), 2A_2+A_1)$, $(D_4(a_1), A_2+3A_1)$, 

\hspace{1cm} $E_8$: ($E_6 + A_1$, $E_8(b_6)$), $(E_8(a_7), D_6(a_2))$,  $(A_3+A_1$, $A_3$). 

In bad characteristic, there are not even the same number of nilpotent and unipotent orbits; for certain bad primes there are more nilpotent orbits than in characteristic $0$. The Hasse diagram for $G_2$ when $p=3$ can be deduced from \cite{Stu71} and is reproduced in Figure \ref{G2Hassediagram}. For the remaining types we will have to work harder to obtain partial information about the closure relations. 

\begin{figure}
\begin{center}
\begin{tikzpicture} [scale=0.5]
  \node (max) at (0,4) {$G_2$};
  \node (a) at (0,2) {$G_2(a_1)$};
  \node (b) at (0,0) {$(\tilde{A_1})_{(3)}$};
  \node (c) at (-2,-2) {$\tilde{A_1}$};
  \node (d) at (2,-2) {$A_1$};
  \node (min) at (0,-4) {$0$};
  \draw (min) -- (d) -- (b) -- (a) -- (max);
  \draw (min) -- (c) -- (b);
\end{tikzpicture}

\caption{Full Hasse diagram for $G_2$ when $p=3$.  \label{G2Hassediagram}}
\end{center}
\end{figure}

We can now prove part of Theorem \ref{thm:mainKos}. 

\begin{lemma} \label{lem:cVisclosed}
The subset $\cV \subseteq \cN$ is a closed $G$-stable subvariety; moreover, it is the closure of a single orbit in each case, as specified in Tables \ref{tab:closurecVclass} and \ref{tab:closurecVexcep}.  
\end{lemma}

\begin{table}
\centering
\begin{tabular}{| c | c | c  c |}  
 \hline
$G$ & $m$ & $\lambda$ &  \\ 
\hline\hline
$A_{m-1}$ & $a(p-1)+r$  & $((p-1)^a,r)$ &  \\
\hline
$B_{\frac{m-1}{2}}$ & $p + a(p-1)+r$ $(r > 0)$ & $(p,(p-1)^{a},r-1,1)$ & $a$ even \\
        & & $(p,(p-1)^{a-1},p-2,r+1)$ & $a$ odd \\
        & $p + a(p-1)$ & $(p,(p-1)^{a})$ &  $a$ even \\
        &  & $(p,(p-1)^{a-1}, p-2,1)$ & $a$ odd \\
        & $\leq p$ & $(m)$ & \\
\hline
$C_{\frac{m}{2}}$ & $a(p-1)+r$  & $((p-1)^a,r)$ & \\
\hline     
$D_{\frac{m}{2}}$  & $p + a(p-1) + r$ & $(p,(p-1)^a,r)$ & $a$ even \\
        & & $(p,(p-1)^{a-1},p-2,r,1)$ & $a$ odd \\
     & $\leq p$  & $(m-1,1)$ &  \\
 \hline
\end{tabular}
\caption{Partition $\lambda$ corresponding to the orbit $\O_{\lambda}$ such that $\cV = \bar{\O}_{\lambda}$ in the classical types, where $a \geq 0$ and $0 \leq r < p-1$.   \label{tab:closurecVclass}}

\end{table} 

\begin{table}
\centering
\begin{tabular}{| c | c | c || c | c | c || c | c | c || c | c | c |}  
 \hline
  $G$ & $p$ & $\O$ &  $G$ & $p$ & $\O$ &  $G$ & $p$ & $\O$ &  $G$ & $p$ & $\O$ \\ 
 \hline\hline
 $G_2$ & $3$ & $\tilde{A}_1^{(3)}$ &  $E_6$ & $3$ &  $A_1^3$  & $E_7$ &  $3$ & $A_1^4$  & $E_8$ & $3$ & $A_1^4$  \\ 
  & $5$ & $G_2(a_1)$ & & $5$ & $D_4(a_1)$ & & $5$ &  $A_3A_2A_1$ & & $5$ & $A_3^2$\\
  & $\geq 7$ & $G_2$ & & $7$ & $E_6(a_3)$& & $7$ & $E_7(a_5)$ & & $7$ & $E_8(a_7)$ \\
  $F_4$ & $3$ & $A_1 \tilde{A_1}$ & & $11$ & $E_6(a_1)$  & &  $11$ & $E_7(a_3)$ & & $11$ & $E_8(a_6)$\\
  & $5$ & $F_4(a_3)$  & & $\geq 13$ & $E_6$ & & $13$ &  $E_7(a_2)$ & & $13$ & $E_8(a_5)$  \\
& $7$ & $F_4(a_2)$ & & & & &  $17$ &  $E_7(a_1)$ & & $17$ &  $E_8(a_4)$ \\
 & $11$ & $F_4(a_1)$ & & & & &$\geq 19$ & $E_7$& & $19$ & $E_8(a_3)$ \\
  & $\geq 13$ & $F_4$ & & & & & & & & $23$ & $E_8(a_2)$ \\
   & & & & & & & & & & $29$ &  $E_8(a_1)$ \\
   & & & & & & & & & & $ \geq 31$ &  $E_8$ \\

 \hline
\end{tabular}
\caption{Orbit $\O$ such that $\cV = \bar{\O}$ in the exceptional types. \label{tab:closurecVexcep}}

\end{table}

\begin{table}
\centering
\begin{tabular}{| c || c  c  c  c  c  c  c  c  |}  
 \hline
$\O$ &  $A_3^2$ &  $D_4(a_1) A_2$ & $A_3 A_2 A_1$ & $A_3A_2$ & $D_4(a_1) A_1$ & $ D_4(a_1)$ & $A_3 A_1^2$ & $A_2^2 A_1^2$ \\

 $\lambda$ & $(5,4^2,1^3) $  & $(5,3^3,1^2)$ & $ (5,3^2,2^2,1)$ & $(5,3^2,1^5)$ & $(5,3,2^2,1^4)$ & $(5,3,1^8)$ & $(5,2^4,1^3)$ & $(3^5,1)$  \\
\hline 
\hline
 $\O$ & $A_3A_1$ & $A_2^2 A_1$ & $A_3$ & $A_2^2$ & $A_2 A_1^3$  & $A_2 A_1^2$ & $A_2 A_1$ & $A_2$\\
 
 $\lambda$ & $(5,2^2,1^7)$ & $(3^4,2^2)$ & $(5,1^{11})$ & $(3^4,1^4)$ & $(3^3,2^2,1^3)$ & $(3^3,1^7)$ & $(3^2,2^2,1^6)$ & $(3^2, 1^{10})$ \\
 \hline 
\hline
$\O$ & $A_1^4$ & $A_1^3$  & $A_1^2$ & $A_1$ & & & &  \\

 $\lambda$ & $(3,2^4, 1^5)$ & $(2^6,1^4)$ & $(2^4,1^{8})$ & $(2^2,1^{12})$ & & & &  \\

\hline

\end{tabular}
\caption{$D_8$ partitions for nilpotent orbits in $\cV$ for $E_8$, $p=5$   \label{tab:Gp3classes}}

\end{table} 
\begin{proof}
Suppose $G=\Cl(V)$ with $\dim V=m$. An orbit corresponding to a partition $\lambda$ of $m$ is contained in the restricted nullcone if and only if the largest part of $\lambda$ is at most $p$. 
Let $G=\SL(V)$ or $\Sp(V)$ (resp. $\SO(V)$), and let $x \in \cN$ with partition represented by $\lambda$. Then $x$ is not regular in a Levi subalgebra with a factor of type $A_{p-1}$ precisely when $\lambda$ contains no parts of size $p$ (resp. at most one part of size $p$). Now every orbit represented in Table \ref{tab:closurecVclass} represents a single orbit in $\cV$: for $G$ of type $D$, each $\lambda$ given in Table \ref{tab:closurecVclass} is not very even. Observe that any other orbit in $\cV$ is represented by a partition lower than $\lambda$ in the dominance ordering, and hence is contained in $\bar{\O}_{\lambda}$; and vice-versa, by definition of $\cV$.

Now suppose $G$ is of exceptional type. We use the tables in the corrected arxiv version of \cite{SteMin} to determine the orbits in the restricted nullcone.  A nilpotent element $x$ is regular in a Levi subalgebra with a factor of type $A_{p-1}$ exactly when the labelling of its orbit contains an $A_{p-1}$ part. Thus in good characteristic, as well as for $G$ of type $G_2$, the result then follows simply by inspecting the Hasse diagrams. 

In the remaining cases we use case-by-case analysis. First let $G$ be of type $E_8$ and $p=5$. Note that every class is distinguished in $\Lie(L)$ for $L$ some Levi subgroup of $G$. Moreover, the Levi subgroups in question are all conjugate to subgroups of $M$, a maximal subgroup of $G$ of type $D_8$. Let $V$ be the $16$-dimensional standard module for $M$. For each non-trivial class in $\cV$ we choose a representative $e$ in $\Lie(M)$ and calculate the Jordan block sizes for the action of $e$ on $V$; these are in Table \ref{tab:Gp3classes}. Note that for some classes there are many non-$M$-conjugate choices for $e$. For example, there are three non-$M$-conjugate Levi subgroups of $M$ of type $A_3^2$; these correspond to the subsets of simple roots $\{1,2,3,5,6,7\}$, $\{1,2,3,5,6,8\}$ and $\{1,2,3,6,7,8\}$. A regular nilpotent element of the corresponding Levi subalgebras will act on $V$ with Jordan blocks of sizes $(4^4)$, $(4^4)$ and $(5,4^2,1^3)$, respectively. 

Note that the final partition is higher in the dominance order than all other partitions in Table \ref{tab:Gp3classes}. Therefore, the closure of the $M$-orbit of a representative of the class $A_3^2$ contains a representative of every class in $\cV$. It remains to prove that there are no more $G$-classes in the closure of the $A_3^2$ class. By \cite[Table~10]{SteMin}, the Jordan block sizes for the adjoint action of nilpotent elements in the $A_3^2$-class are $(5^{38},4^{12},1^{10})$. By embedding $G$ into $\SL_{248}$, it follows that the Jordan block sizes for the adjoint action of every nilpotent element in the closure of the $A_3^2$-class will be lower than $(5^{38},4^{12},1^{10})$ in the dominance order. Using \textit{loc. cit.}, we check that every non-restricted class has a Jordan block of size greater than $5$ and all remaining classes (which have labels with an $A_4$ part) have at least 45 blocks of size $5$. 

Now let $p=3$. When $G$ is of type $F_4$, the subset $\cV$ consists of the zero element and the union of the three classes with labels $A_1$, $\tilde{A}_1$ and $A_1 \tilde{A}_1$. All three non-trivial classes have representatives contained in $\Lie(M)$ where $M$ is a subgroup of type $B_3$. We may choose these representatives so that the corresponding partitions of $7$ are $(2^2,1^3)$, $(3,1^4)$ and $(3,2^2)$, respectively. Therefore, all three classes are contained in the closure of the $A_1 \tilde{A}_1$-class. By \cite[Table~22.1.4]{LS12}, the three classes in $\cV$ for $G$ of type $E_6$ (which are $A_1, A_1^2$ and $A_1^3)$ are all contained in an $F_4$-subalgebra. Therefore the closure of the $A_1^3$-class contains all three classes. 

When $G$ is of type $E_7$, the non-zero elements of $\cV$ consist of the union of the five classes with labels $A_1$, $A_1^2$, $(A_1^3)^{(1)}$, $(A_1^3)^{(2)}$ and $A_1^4$. All such classes have representatives contained in $\Lie(M)$ where $M$ is a subgroup of type $D_6$. We may choose these representatives so that the corresponding partitions of $12$ are $(2^2,1^8)$, $(3,1^9)$, $(2^6)$, $(3,2^2,1^5)$ and $(3,2^4,1)$, respectively. Thus, all the classes in $\cV$ are contained in the closure of the $A_1^4$-class. The discussion in \cite[Section~16.1.2]{LS12} shows that the four non-trivial classes in $\cV$ for $G$ of type $E_8$ (which are $A_1$, $A_1^2$, $A_1^3$ and $A_1^4$) are contained in an $E_7$-subalgebra. Thus the closure of the $A_1^4$-class contains all classes in $\cV$.  

A final routine use of the tables in \cite{SteMin} allows us to complete the proof. For example, when $G$ is of type $E_7$ the Jordan block sizes for the adjoint action of a nilpotent element in the $A_1^4$-class are $(3^{28},2^{14},1^{21})$. Every non-restricted class has a block of size greater than $3$ and all other remaining classes have at least 33 blocks of size $3$.    
\end{proof}

\subsection{\texorpdfstring{$G$}{G}-cr subalgebras}

\begin{Proposition} \label{prop:Gcrovergp}
Suppose $e \in \cN_p$. If $e$ is contained in an $\sl_2$-triple then there exists a $G$-cr subgroup $X \leq G$ of type $A_1$ such that $\Lie(X)$ contains $e$.   
\end{Proposition} 

\begin{proof}
If $G=\SL(V)$ then $e^{[p]}=0$ implies $e$ has Jordan blocks of size at most $p$, which means $e$ is regular in a Levi subalgebra of type $A_{r_1}\times\dots\times A_{r_i}$ with each $r_i\leq p-1$. The image of $X=\SL_2$ under the completely reducible representation given by $L(r_1)\oplus\dots\oplus L(r_i)$ satisfies the demands of the theorem, where $r_j$ now represents a (restricted) high weight. So assume $G$ is not of type $A$. Then if $p$ is good for $G$, it is very good, and the result follows from \cite[Proposition~33,Theorem 52]{McN05}. 

So we may assume $p$ is bad, and therefore that $G$ is exceptional. As before, the orbits of $\cN_p$ can be worked out from the tables in \cite{SteMin} and there are not very many. By inspection, it follows that the label of every restricted nilpotent class is denoted by sums of $A_{r}$ for $r < p$ and $D_4(a_1)$ if $G = E_8$, $p=5$ or is $G_2(a_1)$ when $G= G_2$, $p=3$; note that the class $(\tilde{A}_1)_{(3)}$ is excluded since it is not contained in an $\sl_2$-triple. 

We first deal with the final case. The subsystem subgroup $A_2 < G_2$ contains an $A_2$-irreducible subgroup $X$ of type $A_1$. By \cite[Theorem~1]{SteG2}, all simple subgroups of $G_2$ are $G_2$-cr when $p=3$. The restriction of the nontrivial $7$-dimensional $G_2$-module to $X$ is $L(2)^2 + L(0)$. It follows that the nilpotent elements contained in $\Lie(X)$ have Jordan blocks of size $(3^2,1)$ and thus are in the $G_2(a_1)$ class by \cite[Table~4]{SteMin}. 

In the remaining cases, every class is a distinguished element in $\l = \Lie(L)$ for some Levi subgroup $L$ with simple factors only of type $A_r$ with $r < p$ or $D_4$. By \cite[Proposition 3.2]{Ser05}, a subgroup $X$ of $L$ is $G$-cr if and only if it is $L$-cr. Furthermore a subgroup $X$ of a central product $L = L_1 L_2$ is $L$-cr if and only if the projection of $X$ to both $L_1$ and $L_2$ is $L$-cr. Therefore, it suffices to deal with the cases where $L$ is simple and simply connected of type $A_r$ $(r <p)$ or $D_4$--but these cases have already been tackled.  
\end{proof}

If $X$ is $G$-cr then so is $\Lie(X)$ by \cite[Theorem~1]{McN07}; so we get the following. 

\begin{cor} \label{cor:Gcrsubalgebra}
Suppose $e \in \cN_p$. If $e$ is contained in an $\sl_2$-triple then there exists a $G$-cr subalgebra $\s\cong\sl_2$ of $\g$ containing $e$.     
\end{cor} 

The following is used a couple of times, and is \cite[Lemma 4]{McN07}.

\begin{lemma} \label{L:GcrLcr}
Let $L$ be a Levi factor of a parabolic subgroup of $G$. Suppose that we have a Lie subalgebra $\s \subset \l = \Lie(L)$. Then $\s$ is $G$-cr if and only if $\s$ is $L$-cr.
\end{lemma}

\begin{Proposition}\label{P:SLpE}
Suppose $e\in \cN$ is distinguished in a Levi subalgebra $\l=\Lie(L)$ with a factor of type $A_{p-1}$. Then there is an $\sl_2$-triple $(e,h,f)$ such that $\s: = \langle e,h,f \rangle$ is non-$G$-cr and $f\in \overline{L\cdot e}$.
\end{Proposition}
 
\begin{proof}By Lemma \ref{L:GcrLcr} it suffices to treat the case that $L=\SL(V)$ with $\dim V=p$. In that case, let $\s=\langle e,h,f\rangle$ be the image of $\sl_2$ under the representation given by the $p$-dimensional baby Verma module $Z_0(0)$; cf.~\cite[Section~5.4]{Jan98}. As $V\downarrow X=Z_0(0)$ is a non-trivial extension of the irreducible module $L(p-2)$ by the trivial module we have that $\s$ is not $L$-cr. It is easy to see that one of $e$ or $f$ has a full Jordan block on $V$ and is therefore regular. But the whole of $\cN(L)$ is the closure of a regular nilpotent element so we are done. 
\end{proof}

\begin{lemma} \label{lem:goodcharpsub}
Let $p$ be a good prime for $G$ and $(e,h,f)$ be an $\sl_2$-triple with $e,f \in \cN$. Suppose that $e$ and $f$ are distinguished in Levi subalgebras of $\g$ with no factors of type $A_{p-1}$. If $\s := \langle e,f \rangle$ is $G$-cr then $\s$ is a $p$-subalgebra.
\end{lemma}

\begin{proof}
Suppose $\s$ is not a $p$-subalgebra. Then by \cite[Lemma~4.3]{ST18}, $\s$ is $L$-irreducible in a Levi subalgebra $\l = \Lie(L)$ with $L = L_1 L_2 \ldots L_r$ and $L_1$ of type $A_{rp-1}$, say, for some $r \in \mathbb{N}$. Therefore, the projection $\bar{\s}$ of $\s$ to $\l_1 =\Lie(L_1)$ is also $L_1$-irreducible, so that $\bar\s$ acts irreducibly on the $rp$-dimensional natural $L_1$-module. All irreducible representations of $\sl_2$ have dimension at most $p$ by \cite[Lemma~5.1]{Block}, thus $r=1$. Moreover, the classification of $p$-dimensional irreducible $\sl_2$-modules in \cite[Section~5.4]{Jan98} shows that the image of $e$ or $f$ in $\bar{\s}$ is regular in $L_1$, a contradiction.
\end{proof}

\section{Monogamy of \texorpdfstring{$\cV$}{V}} \label{sec:monogamycV}

We start with an observation that $\cV$ can be characterised using the following partial order on $\cN$. 

\begin{defn}
  Let $x,y \in \cN$. We say $x \preceq y$ (resp.~$x\prec y$) if $\rk (\ad(x)^{p-1}) \le \rk (\ad(y)^{p-1})$ (resp.~$\rk (\ad(x)^{p-1}) < \rk (\ad(y)^{p-1})$).
\end{defn}

Note that $\rk(\ad(x)^{p-1})$ can be calculated from the adjoint Jordan blocks of $x$ of size at least $p$, and if $G$ is exceptional, this can be done by reference to \cite[Section 3.1]{SteMin}. The next lemma follows from  a simple case-by-case check, using Tables \ref{tab:closurecVclass} \& \ref{tab:closurecVexcep}, the Hasse diagrams for nilpotent orbit closures and \cite[Section~3.1]{SteMin}.

\begin{lemma} \label{lem:rankorder}
Let $x,y \in \cN$ such that $x \in \cV$, and $y \notin \cV$. Then $x\prec y$.
\end{lemma}

\begin{remark}Comparing ranks of $(p-1)$-th powers is necessary for the partial order to differentiate nilpotent orbits contained in $\cV$. For example, let be $G$ of type $E_6$, $p=5$, and take $x,y \in \cN$ to be representatives of the $D_4(a_1)$ and $A_4$ classes respectively. Then we have $x \in \cV$ and $y \notin \cV$. Using \cite[Table 16]{SteMin} we see that $\rk (\ad(x)) = \rk (\ad(y)) = 78 $, however $\rk (\ad(x)^{p-1}) = 11 < 15 = \rk (\ad(y)^{p-1})$.\end{remark}

Let $\X\subseteq \cN$. We say that $\X$ is \emph{partially monogamous} if the following holds. 

\begin{center}
Whenever $(e,h,f)$ and $(e,h',f')$ are two $\sl_2$-triples with $e,f, f' \in \X$ and $f, f' \preceq e$, then $f$ and $f'$ are conjugate under the action of $C_G(e)$.
\end{center}

\begin{lemma} \label{l:partialequalsKostant}
Let $\X$ be a subvariety of $\cN_p$. Then $\X$ is monogamous if and only if it is partially monogamous.
\end{lemma}
\begin{proof}
One direction is trivial. Suppose $\X$ is partially monogamous but not monogamous. Then there exist $\sl_2$-triples $(e,h,f)$ and $(e,h',f')$ with $e,f, f' \in \X$ such that $(e,h,f)$ is not $C_G(e)$-conjugate to $(e,h',f')$. Since $\X$ is partially monogamous it follows that either $f \not \preceq e$ or $f' \not \preceq e$; without loss of generality we assume the former. Thus $\rk (\ad(e)^{p-1}) < \rk (\ad(f)^{p-1})$, and in particular, $e$ and $f$ are not conjugate.

Let $(f,\tilde{h},\tilde{e})$ be an $\sl_2$-triple with $f$ conjugate to $\tilde{e}$, which exists by Proposition \ref{prop:Gcrovergp}. Then the two $\sl_2$-triples $(f,-h,e)$ and $(f,\tilde{h},\tilde{e})$ satisfy $f,e,\tilde{e} \in \X$ and $e, \tilde{e} \preceq f$. But as $\X$ is partially monogamous, we have that $f$ is conjugate to $\tilde e$, which is in turn conjugate to $e$, a contradiction.
\end{proof}

Theorem \ref{thm:mainKos} for classical types follows from Lemma \ref{lem:cVisclosed} and the main theorem of \cite{GP}. For the remainder of this section we suppose $G$ is of exceptional type. 

\subsection{Bad characteristic} \label{ss:badchar}

We first treat the case when $p$ is bad for $G$. Fix $0 \neq e \in \cV$ for the remainder of this section. We use the representatives as in \cite{LS12}, presented in \cite{SteMin}. If $G$ is of type $G_2$ and $p=3$, then the element $e$ with label $(\tilde{A}_1)_{(3)}$ cannot be extended to an $\sl_2$-triple by \cite[Theorem 1.7]{ST18}. So we exclude that case from now on. 

\begin{lemma} \label{lem:smoothnorms}
The normaliser $N_G(\langle e \rangle)$ (and centraliser $C_G(e)$) is smooth if and only if the class of $e$ does not occur in the following table.

\begin{table}[ht!]
\centering
\begin{tabular}{| c | c | c |}  
\hline
$G$ & $p$ & class of $e$  \\ 
\hline\hline
$G_2$ & $3$ & $G_2(a_1)$ \\
 $F_4$ & $3$ & $F_4$, $\tilde{A}_2 A_1$ \\
 $E_6$ & $3$ & $E_6$, $E_{6}(a_{1})$, $E_{6}(a_{3})$, $A_{5}$, $A_{2}^2 A_{1}$, $A_2^2$ \\
 $E_8$ & $3$ & $E_{8}$, $E_{8}(a_{1})$, $E_{8}(a_{3})$, $E_{7}$, $E_{6} A_{1}$, $E_{8}(b_{6})$, $A_{7}$, $E_{6}$, $E_{6}(a_{3}) A_{1}$, $A_{5} A_{1}$, $A_{2}^2 A_{1}^2$, $A_{2}^2 A_{1}$ \\
 & $5$ & $E_8$, $A_4 A_3$ \\
  \hline
\end{tabular}
\end{table} 
\end{lemma}

\begin{proof}
Every element $e$ has a cocharacter $\tau$ for which $\im(\tau)$ is contained in $N_G(\langle e \rangle)$ but not $C_G(e)$. Therefore, the dimension of $N_G(\langle e \rangle)$ is precisely $\dim C_G(e) + 1$. Similarly, $\dim \n_\g(\langle e \rangle) = \dim \c_g(\langle e \rangle) + 1$ thanks to the existence of $\sl_2$-triples. Therefore $N_G(\langle e \rangle)$ is smooth precisely when $C_G(e)$ is smooth. 

It is straightforward to use Magma to calculate the dimension of $\c_\g(e)$. Comparing these dimensions with the dimension of $C_G(e)$ presented in \cite[Tables~22.1.1--22.1.5]{LS12} completes the proof.  
\end{proof}

Observe that the set of classes in Lemma \ref{lem:smoothnorms} does not intersect $\cV$, so we may now deduce an important reduction.  

\begin{Proposition} \label{prop:hisfixed}
There exists an $\sl_2$-triple $(e,\bar{h},\bar{f})$ with $\bar{f}$ conjugate to $e$ and $\bar{h} \in \t = \Lie(T)$. Moreover, if $(e,h,f)$ is also an $\sl_2$-triple then $h$ is $C_G(e)$-conjugate to $\bar{h}$. 
\end{Proposition}

\begin{proof}
We know from Proposition \ref{prop:Gcrovergp} that there is an $\sl_2$-triple $(e,\bar{h},\bar{f})$ with $\bar{f}$ in the same nilpotent class as $e$. By Lemma \ref{lem:smoothnorms}, the group $N_G(\langle e \rangle)$ is smooth. Therefore, all maximal tori in $\n_\g(\langle e \rangle)$ are $N_G(\langle e \rangle)$-conjugate. A computation in Magma shows that $\n_\g(\langle e \rangle) \cap \t$ is a maximal torus of $\n_\g(\langle e \rangle)$. So we may assume that $\bar{h}$ is contained in $\t$ (noting that if $(\lambda e, \bar{h}^g, \bar{f}^g)$ is an $\sl_2$-triple then so is $(e, \bar{h}^g,\lambda \bar{f}^g)$). 

For the final part, first note that since $[h,e] = 2e$ we have $[h^{[p]},e] = \ad(h)^p e = 2e$ thanks to Fermat's Little Theorem. Therefore $\h = \langle h^{{[p]}^r} \mid r = 0,1,\ldots \rangle$ is an abelian $p$-closed subalgebra of $\n_\g(\langle e \rangle)$. It follows from \cite[Chapter~2, Corollary~4.2]{SF88} that $\h = \t' \oplus \n'$ where $\t'$ is the set of semisimple elements of $\h$. Since $\t'$ is a torus, the above argument shows that up to $N_G(\langle e \rangle)$-conjugacy we may assume that $\t'$ is contained in $\t$. In particular, $\bar{h} \in \t'$. 

Because $\c_\g(\langle e \rangle )$ has codimension $1$ in $\n_\g(\langle e \rangle )$ and $\bar{h} \not \in \c_\g(\langle e \rangle )$ we see that the torus $\t'$ decomposes as $\t' = \c_{\t'}(e) \oplus \langle \bar{h} \rangle$. Furthermore, $\n' \subset \c_\g(\langle e \rangle)$. It follows that $h = \bar{h} + h'$ for some $h' \in \c_\g(e) \cap \c_\g(\bar{h})$. 

Since $h = [e,f]$ and $\bar{h} = [e,\bar{f}]$ we also have $h' \in \im(\ad(e))$. Thus
$$h' \in W = \c_\g(\langle e, h \rangle)  \cap \im(\ad(e)).$$ 
Another Magma check shows that every element in $W$ is $p$-nilpotent. 

In particular, all eigenvalues of $h'$ are $0$. Since $h = \bar{h} + h'$ and $[h,f] = -2f$ we must have $[\bar{h},f] = -2f$. Therefore, $f \in F = \ker(\ad(\bar{h}) + 2I_{\dim \g})$ and so $h = [e,f] \in \im(\ad(e))(F)$. Note that $\bar{f} \in F$ also, so $\bar{h} \in \im(\ad(e))(F)$ and hence $h' \in \im(\ad(e))(F)$. 

Thus $h' \in W \cap \im(\ad(e))(F)$. A final easy check in Magma shows that $W \cap \im(\ad(e))(F) = 0$, as required.  
\end{proof}

We now describe an ad-hoc method to prove that if $(e,h,f')$ is an $\sl_2$-triple with $f' \in \cV$ and $f' \preceq e$ then $f'$ is uniquely determined up to $C:= (C_G(e) \cap C_G(h))$-conjugacy. In principle, this can be implemented by hand, but for speed and accuracy we have used Magma. Applying Proposition \ref{prop:hisfixed} and Lemma \ref{l:partialequalsKostant} then completes the proof that $\cV$ is monogamous. 

\underline{Setup:}

By Proposition \ref{prop:hisfixed}, there exists an $\sl_2$-triple $(e,h,f)$ with $h \in \t = \Lie(T)$ and $f \in \cV$ in the same nilpotent class as $e$. Let $(e,h,f')$ be an $\sl_2$-triple with $f' \in \cV$ and $f' \preceq e$. Since \begin{equation}[h,f']=-2f'\label{c1}\end{equation} we have $f' \in F:=\ker(\ad(h) + 2I_{\dim(\g)})$. We set up a generic element of the subspace $F$, namely $\tilde f = \sum x_i v_i \in \g$ where the $x_i$ are variables and $v_1, \ldots, v_{\dim (F)}$ is a basis for $F$. One can view the set of all $\tilde{f}$ as describing a subvariety $\cf$ of $\g$. In Steps 1 to 3 below, we add in additional equations and
thus replace $\cf$ with successively smaller sets (still called $\cf$ by abuse of notation).

\underline{Step 1:} The equation \begin{equation}[e,\tilde{f}]= h\label{c2}\end{equation} yields a set of linear equations among the $x_i$. We use these to constrain $\tilde{f}$ and thus reduce the dimension of $\cf$. Now every element of $\cf$ forms an $\sl_2$-triple with $e$. 

\begin{example}
We give an example where Step 1 is sufficient. Let $G$ be of type $E_7$, $p=3$ and $e = e_{\alpha_2} + e_{\alpha_5} + e_{\alpha_7}$. Then $e$ is a representative of the $(A_1^3)^{(1)}$ orbit and $e \in \cV$ by Lemma \ref{lem:cVisclosed}. On this occasion it is obvious that $(e,h,f)$ is an $\sl_2$-triple with $h=h_2 + h_5 + h_7 \in \t$ and $f= e_{-\alpha_2} + e_{-\alpha_5} + e_{-\alpha_7}$.

Let $F:=\ker(\ad(h) + 2I_{\dim(\g)})$. A straightforward calculation shows that the space $F$ is $27$-dimensional with a basis of root vectors $v_1 = e_{r_1}, \ldots, v_{27} = e_{r_{27}}$ for some set of roots $r_1, \ldots, r_{27}$; 
in particular $r_{12} = -\alpha_2$, $r_{13} = -\alpha_5$ and $r_{14} = -\alpha_7$.  

We let $\tilde f = \sum_i x_i v_i$ as above. We then compute $[e,\tilde{f}] = h$. For $i \neq 12,13,14$ we find that the left hand side has a coordinate of the form $\lambda x_i$ for $\lambda = 1$ or $2$. Thus $x_i = 0$ for $i \neq 12,13,14$. On the other hand the coordinate of $h_2$ is seen to be equal to $x_{14} + 2$. Thus $x_{14}$ is $1$. Similarly, the coordinates of $h_5$ and $h_7$ are $x_{13}+2$ and $x_{12}+2$, respectively. We have therefore determined all the variables in $\tilde{f}$ and in fact $\tilde{f} = f$, which is sufficient.     
\end{example}

\underline{Step 2:} The adjoint action of $C$ preserves $\cf$. Find a set of variables $\{ x_i \mid i \in Z\}$ such that every $C$-orbit in $\cf$ contains a representative with $x_i = 0$ for $i \in Z$. Thus we may assume that these variables are zero in $\tilde{f}$, further reducing $\cf$.

\begin{example}
We give an example where Steps 1 and 2 are sufficient. Let $G$ be of type $G_2$ and $p =3$. Consider $e = e_{10}$ which is a representative of the $\tilde{A_1}$ orbit, thus contained in $\cV$ by Lemma \ref{lem:cVisclosed}.  

Clearly, if $h = h_{10}, f= e_{-10}$, then $(e,h,f)$ is an $\sl_2$-triple with $f \in \cV$. Define $F:=\ker(\ad(h) + 2I_{\dim(\g)})$. This is 3-dimensional and we build $\tilde f$ as above: 
\[\tilde f = x_1e_{-11}+x_2e_{-10}+x_3e_{21}.\]

After Step 1 we find \[\tilde f = x_1e_{-11}+e_{-10}+x_3e_{21}.\]

Now we apply elements of $C = C_G(e) \cap C_G(h)$ to $\tilde f$. First consider $x_{-01}(t) \in C$. We calculate that
\[ x_{-01}(t) \cdot \tilde f = (t + x_1)e_{-11} +e_{-10}+x_3e_{21}. \]
Therefore, by setting $t = -x_1$, we see that every $C$-orbit in $\cf$ contains a representative with $x_1 = 0$. 
We're down to \[\tilde{f} = e_{-10}+x_3e_{21}.\] 
Finally, conjugation by $x_{31}(t) \in C$ sends $\tilde{f}$ to $e_{-10} + (t + x_3) e_{21}$. Thus we conclude that $\tilde{f} = f$, as required.    
\end{example}

\underline{Step 3:} Finally, we impose the condition that $\tilde{f}$ should represent an element $f' \in \cV$ with $f' \preceq e$. Since every element in $\cV$ is $p$-nilpotent, the equation  
\begin{equation}\ad(\tilde{f})^p = 0.\label{c3}\end{equation} yields further polynomial equations we want the $x_i$ to satisfy.

Forcing $\cf$ to only contain elements $f'$ with $f' \preceq e$ is slightly more subtle since we cannot simply calculate the `rank' of $M = \ad(\tilde{f})^{p-1}$. Let $R = \text{rank}(\ad(e)^{p-1})$ and $\epsilon$ be a map evaluating the remaining variables to choices in $\kk$ (so each $f' \in \cf$ is simply some $\epsilon(\tilde{f})$). We find a subset $r_1, \ldots, r_R$ of rows and subset $c_1, \ldots, c_R$ of columns such that, up to the reordering of rows and columns, the corresponding submatrix $S$ of $M$ is upper triangular and all diagonal entries are elements of $\mathbb{F}_p^*$. Then any element $f' \in \cf$ will satisfy $\rk(\ad(f')^{p-1}) \geq R$. We only want those elements $f' \preceq e$ which means $\rk(\ad(f')^{p-1}) \leq R$. Thus, given any row $r$ of $M$ the element $\epsilon(r)$ is in the span of $\epsilon(r_1), \ldots, \epsilon(r_R)$. In particular, a row $r'$ of $M$ with zeroes at all columns $c_1, \ldots, c_R$ evaluates to zero. This final set of conditions is enough to force all remaining variables to be $0$.

\begin{example}
We give an example where we require Step 3. Let $G$ be of type $G_2$ and $p =3$. Consider $e = e_{01}$ which is a representative of the $A_1$ orbit, thus contained in $\cV$ by Lemma \ref{lem:cVisclosed}.  

Take $h = h_{01}, f= e_{-01}$, then $(e,h,f)$ is an $\sl_2$-triple in $\g$ with $f \in \cV$. Define $F:=\ker(\ad(h) + 2I_{\dim(\g)})$. This is 5-dimensional and we build $\tilde f$ as above: 
\[\tilde f = x_1e_{-32}+x_2e_{-01}+x_3e_{-10}+x_4e_{11}+x_5e_{32}.\]

After Step 1 we find \[\tilde f = x_1e_{-32}+e_{-01}+x_3e_{-10}+x_4e_{11}+x_5e_{32}.\]

There are no elements of $C = C_G(e) \cap C_G(h)$ which we can use to reduce $\tilde f$, so we move onto Step 3. 

The equation $\ad(\tilde{f})^p = 0$ gives many relations amongst the remaining variables but none that allow us to conveniently reduce $\tilde{f}$. Consider the matrix $M = \ad(\tilde{f})^{p-1}$. The first, eighth, tenth and thirteenth column of $M$ consist only of zeroes, so we remove them, leaving the matrix $M'$ as follows.
\[\small\begin{pmatrix}
  x_1x_5 & 0 & 0 & x_5 & 2x_4^2 & 0  & 0 & x_4x_5 & 0 &  x_5^2 \\
  0 & 2x_4 & 0 & 0 & 0 & x_5  & 0 & 0 & 0 & 0 \\
  0 & 0 & 2x_4 & 0 & 0 & 0 & x_5 & 0 & 0 & 0  \\
  0 & 0 & 0 & 0 & 2x_1x_5+x_3x_4 & 0 & 0 & x_3x_5+x_4^2 & 0  & 0\\
  0 & 2x_1x_4+2x_3^2 & 0 & 0 & 0 & x_1x_5+2x_3x_4 & 0 & 0 & x_3x_5+x_4^2 &  0 \\
  0 & x_3 & 0 & 0 & 0 & 2x_4 & 0 & 0 & x_5 & 0  \\
  0 & 0 & x_3 & 0 & 0 & 0 & x_4 & 0 & 0 & 0  \\
  0 & 0 & 0 & 0 & 0 & 0 & 0 & 0 & 0 & 0  \\
  0 & 0 & 0 & 0 & x_1x_4 + x_3^2 & 0 & 0 & 2x_1x_5 + x_3x_4 & 0  & 0 \\
  x_1 & 0 & 0 & 1 & x_3 & 0 & 0 & x_4 & 0 & x_5 \\
  0 & 2x_1 & 0 & 0 & 0 & 2x_3 & 0 & 0 & 2x_4  & 0 \\
  0 & 0 & 2x_1 & 0 & 0 & 0 & x_3 & 0 & 0 & 0 \\
 x_1^2 & 0 & 0 & x_1 & x_1x_3 & 0 & 0 & 2x_3^2 & 0  & x_1x_5 \\
 0 & 0 & 0 & 0 & 0 & x_1 & 0 & 0 & x_3 & 0
  \end{pmatrix}\]

We calculate that $R = \text{rank}(\ad(e)^{p-1}) =1$. Therefore, if $\epsilon(\tilde{f}) = f' \preceq e$ for some evaluation map $\epsilon$, the rank of $\epsilon(M')$ is at most one. Observe that $M'_{10,4}=1$ and so the rank of $\epsilon(M')$ is at least one. It follows that every row of $\epsilon(M')$ is a multiple of the tenth row of $\epsilon(M')$. 

Consider the sixth row of $M'$. This only has nonzero entries in columns 2, 6 and 9, namely $x_3, 2x_4$ and $x_5$. Since the tenth row is zero in columns 2, 6 and 9, the sixth row of $\epsilon(M')$ is zero. Hence $x_3=x_4=x_5=0$.

Similarly, row 11 of $\epsilon(M')$ is zero. Thus $x_1 =0$, and we conclude that $\tilde{f} = f$.
\end{example}

\subsection{Good characteristic} \label{ss:goodchar}

Suppose $p$ is a good prime for $G$. As in the bad characteristic case, we describe an algorithm to deduce that $\cV$ is monogamous. In good characteristic there is a considerable amount of theory at our disposal. In particular, every $e\in\cN$ has an associated cocharacter: that is a homomorphism $\tau:\G_m\to G$ such that under the adjoint action, we have $\tau(t)\cdot e=t^2e$ and $\tau$ evaluates in the derived subgroup of the Levi subgroup in which $e$ is distinguished.

\begin{lemma} \label{lem:hunique}
Suppose $p$ is good for $G$, and let $(e,h_1,f_1)$ be an $\sl_2$-triple with $e,f_1 \in \cV$. Then there exists a cocharacter $\tau$ associated to $e$ such that $\Lie(\tau(\G_m)) = \langle h_1 \rangle$. Thus if $(e,h_2,f_2)$ is also an $\sl_2$-triple with $f_2 \in \cV$, then $h_2$ is $C_G(e)$-conjugate to $h_1$. Moreover, if $h_1 = h_2$ and $\g = \bigoplus_i \g(i)$ is the grading of $\g$ with respect to $\tau$ we have
\[ f_1 - f_2 \in \bigoplus_{r > 0} \g_e(-2+rp), \]
where $\g_e(i):= \c_\g(e) \cap \g(i)$. 
\end{lemma}

\begin{proof}
We start by proving that $h_i$ is toral. By Lemma \ref{lem:goodcharpsub}, the subalgebra $\s_i = \langle e,h_i,f_i \rangle$ is either a $p$-subalgebra or non-$G$-cr. In the former case, we are done. In the latter case, the argument in the proof of \cite[Lemma~6.1]{ST18} applies, showing $h_i$ is toral. 

Now we apply \cite[Proposition~2.8]{ST18}. This yields cocharacters $\tau_i$ associated to $e$ such that $\Lie(\tau_i(\G_m)) = \langle h_i \rangle$. By \cite[Lemma~5.3]{Jan04}, any two cocharacters associated to $e$ are $C_G(e)$-conjugate. Therefore, $h_1$ and $h_2$ are $C_G(e)$-conjugate and so up to $C_G(e)$-conjugacy we may assume they are equal. Set $h = h_1 = h_2$. 

Since $[e,f_1-f_2] = h - h = 0$ we know $f_1 - f_2 \in \c_\g(e)$. Furthermore, $[h,f_1-f_2] = -2(f_1 - f_2)$ and hence $f_1-f_2 \in \bigoplus_r \g(-2 + rp)$. The conclusion follows by noting that $\c_\g(e)$ is contained in the nonnegative graded part of $\g$. 
\end{proof}

Fix $0 \neq e \in \cV$ for the remainder of this section. Choose a cocharacter $\tau$ associated to $e$ such that $h \in \Lie(\tau(\G_m)) \subset \t$ with $[h,e]=2e$. In practice, we use the representatives and associated cocharacters given in \cite{LT11}. We know from Pommerening \cite{Pom1,Pom2} and Lemma \ref{lem:hunique} that there exists a unique $\bar{f} \in \g(-2)$ such that $(e,h,\bar{f})$ is an $\sl_2$-triple. Furthermore, if $(e,h,f)$ is another $\sl_2$-triple then $f = \bar{f} + f'$ with $f' \in \bigoplus_{r > 0} \g_e(-2+rp)$. Therefore, we need to prove that if $f \in \cV$ then up to $C = C_G(e) \cap C_G(h)$-conjugacy we have $f = \bar{f}$, i.e. that $f' = 0$.  

To do this we use the ad-hoc method from Section \ref{ss:badchar}. Indeed, by Lemma \ref{l:partialequalsKostant} it suffices to prove that $f = \bar{f}$ when $f \preceq e$. We now apply Steps 1--3 starting with the space $F = f + \bigoplus_{r > 0} \g_e(-2+rp)$.  

\begin{example}
We give a final example, this time in good characteristic. Let $G$ be of type $E_7$ and $p =7$. Consider $e = e_{\substack{1 \\ \ } \substack{0 \\ \ } \substack{0 \\ 0} \substack{0 \\ \ } \substack{0 \\ \ } \substack{0 \\ \ } }+ e_{\substack{0 \\ \ } \substack{1 \\ \ } \substack{0 \\ 0} \substack{0 \\ \ } \substack{0 \\ \ } \substack{0 \\ \ }} + e_{\substack{0 \\ \ } \substack{0 \\ \ } \substack{1 \\ 0} \substack{0 \\ \ } \substack{0 \\ \ } \substack{0 \\ \ }} + e_{\substack{0 \\ \ } \substack{0 \\ \ } \substack{0 \\ 0} \substack{1 \\ \ } \substack{0 \\ \ } \substack{0 \\ \ } } + e_{\substack{0 \\ \ } \substack{0 \\ \ } \substack{0 \\ 0} \substack{0 \\ \ } \substack{1 \\ \ } \substack{0 \\ \ } }$ which is  a representative of the $(A_5)^{(2)}$ orbit; thus $e \in \cV$ by Lemma \ref{lem:cVisclosed}. Furthermore, by \cite[p.~109]{LT11}, $e$ has an associated cocharacter with the following $\tau$-weights on simple roots {\arraycolsep=2pt $\tau = {\Small\begin{array}{c c r c c c c}2&2&2&2&2&-5\\&&-9\end{array}}$}. One uses the inverse of the Cartan matrix to convert this into a sum of coroots, yielding $h=2h_1+6h_3+5h_4+6h_5+2h_6\in \Lie(\tau(\mathbb{G}_m))$ (this process is how one gets from the diagram of the distinguished cocharacters in Section~11 to the cocharacters given in Table~3 of \textit{ibid.}). The unique $\bar{f} \in \g(-2)$ such that $(e,h,\bar{f})$ is an $\sl_2$-triple is then given by $\bar{f} = 2e_{\substack{-1 \\ \ } \substack{0 \\ \ } \substack{0 \\ 0} \substack{0 \\ \ } \substack{0 \\ \ } \substack{0 \\ \ } }+ 6e_{\substack{-0 \\ \ } \substack{1 \\ \ } \substack{0 \\ 0} \substack{0 \\ \ } \substack{0 \\ \ } \substack{0 \\ \ }} + 5e_{\substack{-0 \\ \ } \substack{0 \\ \ } \substack{1 \\ 0} \substack{0 \\ \ } \substack{0 \\ \ } \substack{0 \\ \ }} + 6e_{\substack{-0 \\ \ } \substack{0 \\ \ } \substack{0 \\ 0} \substack{1 \\ \ } \substack{0 \\ \ } \substack{0 \\ \ } } + 2e_{\substack{-0 \\ \ } \substack{0 \\ \ } \substack{0 \\ 0} \substack{0 \\ \ } \substack{1 \\ \ } \substack{0 \\ \ } }$. 

Let $F = f + \bigoplus_{r > 0} \g_e(-2+rp)$, which is $6$-dimensional. We build a generic element $\tilde{f}$ of $F$ as in Section \ref{ss:badchar} with six variables. Following Step 1 by enforcing the linear equations from $[e,\tilde{f}] = h$ yields 
\[ \tilde{f} = \bar{f} + x_1e_{\substack{-1 \\ \ } \substack{2 \\ \ } \substack{3 \\ 2} \substack{2 \\ \ } \substack{1 \\ \ } \substack{1 \\ \ } }+x_2e_{\substack{-0 \\ \ } \substack{0 \\ \ } \substack{1 \\ 1} \substack{1 \\ \ } \substack{0 \\ \ } \substack{0 \\ \ } } +x_2 e_{\substack{-0 \\ \ } \substack{1 \\ \ } \substack{1 \\ 1} \substack{0 \\ \ } \substack{0 \\ \ } \substack{0 \\ \ } }+ x_3e_{\substack{-0 \\ \ } \substack{0 \\ \ } \substack{0 \\ 0} \substack{0 \\ \ } \substack{0 \\ \ } \substack{1 \\ \ } } + x_4 e_{\substack{1 \\ \ } \substack{1 \\ \ } \substack{1 \\ 0} \substack{1 \\ \ } \substack{1 \\ \ } \substack{1 \\ \ } }- x_5e_{\substack{1 \\ \ } \substack{2 \\ \ } \substack{2 \\ 1} \substack{1 \\ \ } \substack{1 \\ \ } \substack{0 \\ \ } }+x_5e_{\substack{1 \\ \ } \substack{1 \\ \ } \substack{2 \\ 1} \substack{2 \\ \ } \substack{1 \\ \ } \substack{0 \\ \ } }+x_6e_{\substack{2 \\ \ } \substack{3 \\ \ } \substack{4 \\ 2} \substack{3 \\ \ } \substack{2 \\ \ } \substack{1 \\ \ } }.\] 

On this occasion $C:=C_G(e) \cap C_G(h)$ is finite and we move onto Step 3. 

Let $M = \ad(\tilde{f})^{p-1}$. We calculate that $R = \text{rank}(\ad(e)^{p-1}) =13$. So if $\epsilon(\tilde{f}) = f' \preceq e$ for some evaluation map $\epsilon$, we have that the rank of $\epsilon(M)$ is at most 13. 

Ordering the basis of $\g$ as in Magma, we use the $13 \times 13$ submatrix $S$ of $M$ corresponding to the rows $r$ and columns $c$ where \begin{align*}r &= \{75, 125, 62, 94, 87, 129, 120, 97, 42, 82, 23, 34, 108 \}, \\ c&=\{37, 100, 24, 52, 50, 109, 92, 60, 14, 40, 5, 9, 72\}.\end{align*} 
The submatrix $S$ is upper triangular and all diagonal entries are elements of $\mathbb{F}_p^*$. The only other nonzero entries in $S$ can be found in row one, which is
\[ (1  \ \ 0  \ \ 4x_2 \ \  0  \ \ 0 \ \ 0 \ \ 5x_5 \ \ 0 \ \  0 \ \ 0 \ \ 0 \ \ 0 \ \ 0).\]
We find that 42 rows of $M$ have zero entries in every column in $c$, so each of these rows is zero. An example of such a row is the eighth row of $M$. In row 8 we find $x_4, 3x_5$ and $-x_6$ in columns $11,15$ and $70$ respectively. It follows that $x_4 = x_5 = x_6 = 0$. Similarly the 133rd row of $M$ then allows us to deduce that $x_1=x_2=x_3 = 0$. Thus $\tilde{f} =f $ as required.
\end{example}

\section{Proof of Theorems \ref{thm:mainKos} and \ref{thm:mainGcr}} \label{sec:thmproofs}

Proposition \ref{prop:Gcrovergp} shows that for each $e \in \cV$ there exists an $\sl_2$-triple $(e,h,f)$ with $\s = \langle e,h,f \rangle = \Lie(X)$ for a $G$-cr subgroup $X < G$ of type $A_1$. Thus $f$ is $G$-conjugate to $e$ and hence $f \in \cV$. We have demonstrated in Section \ref{sec:monogamycV} that any other $\sl_2$-triple $(e,h',f')$ with $f' \in \cV$ is $C_G(e)$-conjugate to $(e,h,f)$. Therefore $\s' = \langle e,h',f' \rangle$ is $G$-conjugate to $\s$ and hence $G$-cr. 

It remains to prove that $\cV$ is the unique maximal closed $G$-stable subvariety of $\cN$ satisfying both the monogamy and $A_1$-$G$-cr conditions. 

For $G$ of classical type, it follows from \cite[Theorem 1.1]{GP} that $\cV$ is maximal with respect to being monogamous and the unique subvariety with this property. For the $A_1$-$G$-cr property, the ingredients are in \emph{ibid.} but let us spell out the details, as these essentially make up the strategy for the groups of exceptional type used below.

\begin{Proposition} \label{prop:classicalmain}
Let $G$ be a simple algebraic group of classical type. Then $\cV$ is the unique maximal closed $G$-stable $A_1$-$G$-cr subvariety of $\cN$. 
\end{Proposition}

\begin{proof}
Suppose $\X$ is a $G$-stable closed subvariety of $\cN$ satisfying the $A_1$-$G$-cr condition and $\X \not\subseteq\cV$. Let $e \in \X \setminus \cV$. 

First suppose $e$ is distinguished in a Levi subalgebra $\l = \Lie(L)$ with $L$ having a factor of type $A_{p-1}$. Proposition \ref{P:SLpE} shows that $e$ is contained in an $\sl_2$-triple generating a non-$G$-cr subalgebra, a contradiction (these non-$G$-cr subalgebras are also exhibited in \cite[Section~2.4]{GP}).

By definition of $\cV$, we may now assume that $e^{[p]} \neq 0$. The discussion before Proposition 2.2 in \emph{ibid.} exhibits an $\sl_2$-triple $(e,h,f)$ with $f^{[p]}=0$ and $f$ in $\bar{G \cdot e}$, thus $f \in \X$. The argument in the first paragraph shows that neither $e$ nor $f$ are distinguished in a Levi subalgebra with a factor of type $A_{p-1}$. By Lemma \ref{lem:goodcharpsub}, the $\sl_2$-subalgebra $\langle e, f\rangle$ is non-$G$-cr, a final contradiction.   
\end{proof}

\begin{Proposition} \label{P:uniquemax}
Let $G$ be a simple algebraic group of exceptional type. The variety $\cV$ is the unique maximal closed $G$-stable subvariety of $\cN$ satisfying both the monogamy and $A_1$-$G$-cr conditions. 
\end{Proposition} 

\begin{proof}
Suppose $\X$ is a $G$-stable closed subvariety of $\cN$ satisfying either the monogamy or $A_1$-$G$-cr condition and $\X \not\subseteq\cV$. 

First suppose there exists $e \in \X$ which is distinguished in a Levi subalgebra $\l = \Lie(L)$ with a factor of type $A_{p-1}$. Then Propositions \ref{prop:Gcrovergp} and \ref{P:SLpE} furnish us with two $\sl_2$-triples $(e,h,f)$ and $(e,h',f')$ such that the first generates a $G$-cr subalgebra and the second generates a non-$G$-cr subalgebra. Moreover, $f$ is in the same $G$-class as $e$ and $f'$ is in the closure of the $G$-class of $e$. Hence $\X$ does not satisfy either condition, a contradiction.

Thus, we now assume every element of $\X$ is distinguished in a Levi subalgebra with no factors of type $A_{p-1}$. Since $\X \not\subseteq \cV$, there exists a nilpotent class in $\X$ with representative $e$ distinguished in a Levi subalgebra $\l = \Lie(L)$ of $\g$ such that $e^{[p]} \neq 0$. 

Suppose $p$ is good for $L$. From \cite[Section~2.4]{PremetStewart} we find an $\sl_2$-triple $(e,h,f)$ of $\l$ with $f^{[p]} = 0$. Since $p$ is good for $L$, we may simply inspect the Hasse diagrams of each factor of $L$ to deduce that every restricted nilpotent class is contained in the closure of each non-restricted distinguished class. Thus, $f \in X$. Furthermore, $\s = \langle e,f \rangle \cong \sl_2$ is a non-$L$-cr subalgebra by Lemma \ref{lem:goodcharpsub}. Hence by Lemma \ref{L:GcrLcr}, $\X$ does not satisfy the $A_1$-$G$-cr condition. Proposition \ref{prop:Gcrovergp} yields an $\sl_2$-triple $(f,h',e')$ which generates a $G$-cr $\sl_2$-subalgebra, and moreover $e'$ is in the same $G$-class as $f$. Therefore, $f$ is contained in two non-conjugate $\sl_2$-triples. Thus $\X$ does not satisfy the monogamy condition either. 

In the remaining cases $p$ is bad for $L$ (and hence for $G$) so $L$ has an exceptional factor (including the cases $L = G$). For each class, we choose $e$ to be the representative as in \cite{LS12}. Then \cite[Theorem~1~(iii)(b)]{LS12} provides a parabolic subgroup $P = QL$ of $G$ and a $1$-dimensional torus $T_1 < Z(L)$ with the following properties. Let $Q_{\geq 2}$ be the product of all root groups for which the $T_1$-weight is at least $2$. Then $e \in \q_{\geq 2}:=\Lie(Q_{\geq 2})$ and moreover, the closure of the $P$-orbit of $e$ is equal to $\q_{\geq 2}$. Thus, $\q_{\geq 2} \subseteq \X$. Unless $G$ is of type $G_2$ (this case is dealt with momentarily), a straightforward calculation shows that $\q_{\geq 2}$ contains a representative of the $A_{p-1}$-class. Thus, so does $\X$, which is a contradiction. 

Finally, let $G$ be of type $G_2$ and $p=3$. The only two classes not contained in $\cV$ are the regular and the subregular. Since the closure of the regular class contains the subregular class it suffices to assume $\X$ contains the subregular class. A representative for this orbit is $e = e_{\alpha_2} + e_{-3\alpha_1 - \alpha_2}$. This is a regular nilpotent element in $\m = \Lie(M)$ where $M$ is the standard subsystem subgroup of type $A_2$ corresponding to the simple roots $\alpha_2$ and $-3\alpha_1 - 2\alpha_2$.

As in the proof of Proposition \ref{P:SLpE}, there exists an $\sl_2$-triple $(e,h,f)$ in $\m$ such that $\s = \langle e, f \rangle$ is non-$M$-cr. Furthermore, $f$ is in the orbit labelled $A_1$ (both as an $A_2$-orbit and $G_2$-orbit). We claim that $\s$ is non-$G$-cr. By Proposition $\ref{prop:Gcrovergp}$, the element $f$ is contained in an $\sl_2$-triple generating a $G$-cr subalgebra and by the claim, the $\sl_2$-triple $(f,-h,e)$ generates a non-$G$-cr subalgebra. Hence $\X$ does not satisfy either condition. 

For the claim, note that $\s$ is certainly $G$-reducible since it is non-$M$-cr. All $G$-cr $\sl_2$-subalgebras which are $G$-reducible are contained in a Levi subalgebra. In this low-rank case, it immediately follows that all such $\sl_2$-subalgebras are $G$-conjugate to either $\l_1 = \langle e_{\pm \alpha_1} \rangle$ or $\l_2 = \langle e_{\pm \alpha_2} \rangle$. Therefore a $G$-cr $\sl_2$-subalgebra only contains nilpotent elements in the $A_1$ or $\tilde{A}_1$ classes. The claim follows since $\s$ contains $e$ which is in the subregular class.
\end{proof}

{\footnotesize
\bibliographystyle{amsalpha}
\bibliography{bib}}

\end{document}